\documentclass[a4paper,12pt]{article}
\usepackage[top=2.5cm,bottom=2.5cm,left=2.5cm,right=2.5cm]{geometry}
\usepackage{graphicx,xcolor}
\usepackage{amsthm,amsmath,amsfonts,dsfont,cancel,tensor}

\usepackage[
	colorlinks=true,
	linkcolor=violet,
	urlcolor=blue,
	citecolor=teal]{hyperref}

\newtheorem{defn}{Definition}
\newtheorem{lem}{Lemma}
\newtheorem{thm}{Theorem}

\newtheorem{cor}{Corollary}
\newtheorem{rmk}{Remark}
\newtheorem{ex}{Example}

\usepackage{xspace}
\newcommand{\HMF}{Hesse-Frobenius\xspace}
\newcommand{\CMF}{Curved Frobenius\xspace}
\newcommand{\Nabla}{\mathsf{D}}
\newcommand{\Ass}{\mathrm{Assoc}}
\newcommand{\As}{\mathsf{A}}
\newcommand{\sign}{\mathrm{sign}}

\usepackage[sortcites=true,giveninits=true]{biblatex}
\title{Hessian geometry and\\ Frobenius manifolds with curvature}
\author{\textsc{Andreas Vollmer}\\
        \fontsize{10pt}{12pt}\selectfont 
        \textit{University of Hamburg},
        \textit{Department of Mathematics},
        \textit{Hamburg}}
\date{\today}

\begin{document}

\maketitle

\begin{abstract}	
    A \emph{Hessian manifold} is a Riemannian manifold such that, locally, the metric can be written as the Hessian of a function, which is called the Hessian potential.
    A \emph{\CMF manifold} is a Riemannian manifold furnished with a commutative product that is compatible with the metric. In addition, the product's associator has to be controlled by the curvature tensor and a certain potentiality property must hold.
    
    This paper considers the interplay of Hessian and \CMF geometry.
    We obtain that \CMF manifolds of constant sectional curvature are consistent with a Hessian structure, if they satisfy a non-linear closed prolongation system of finite type. Locally, they are therefore characterised by purely algebraic initial data.
    Consistency here means that the Frobenius potential and the Hessian potential can be identified.
    
    As an application, we show that certain second-order maximally superintegrable systems locally correspond 1-to-1 to \CMF structures that are consistent with a Hessian structure.
\end{abstract}



\section{Introduction}

This paper juxtaposes two geometric structures: on the one hand, \emph{Hessian metrics} naturally arise in information geometry as the flat case of statistical manifolds. A statistical manifold is a Riemannian manifold $(M,g)$ equipped with a pair of torsion-free connections $\Nabla,\Nabla^*$ that are dual with respect to the metric $g$,
\begin{equation}\label{eq:dual.connection}
	Z(g(X,Y)) = g(\Nabla_ZX,Y)+g(X,\Nabla^*_ZY)
\end{equation}
for vector fields $X,Y,Z$, cf.~\cite{Nielsen2020,Shima2007}.
If both $\Nabla$ and $\Nabla^*$ are flat, the metric $g$ can be written, locally, as the Hessian of the a function. A proper definition is given below.
A famous example of a Hessian manifold arises from the Fisher information matrix, which leads to the Fisher-Rao metric \cite{Fisher1922,Rao1945,Hotelling1930}. This metric has important applications, such as in machine learning, biomathematics and thermodynamics. In equilibrium statistical mechanics, for example, it leads to the Ruppeiner and the Weinhold metric \cite{Ruppeiner1979,Ruppeiner1996,Weinhold1976}.

On the other hand, Frobenius manifolds have been introduced by Dubrovin in the context of $2$-dimensional topological field theories \cite{Dubrovin1996}. They have wide applications, for instance in cohomology and singularity theory. Conceptually speaking, Frobenius manifolds are flat Riemannian manifolds equipped with a commutative and associative product that is compatible with the underlying metric. Dubrovin's original definition also requires the existence of a unit element and of an Euler vector field. These are not required here, following a more general approach by Manin \cite{Manin1999}. In fact, we also admit non-vanishing curvature, following the ideas in \cite{Mokhov2008,KKLNS2017,KKLNS2018}, for example.
As a consequence, the products can be non-associative, with the associator controlled by the curvature tensor.
\bigskip

Let $(M,g)$ be a Riemannian manifold of dimension $n\geq2$.
We denote the bundle of vector fields on $M$ by $\mathfrak X(M)$ and the Levi-Civita connection of $g$ by $\nabla$.
For a product $\star:\mathfrak X(M)\times\mathfrak X(M)\to\mathfrak X(M)$, its \emph{associator} is defined by
\[
    \Ass_\star(X,Y,Z)=(X\star Y)\star Z-X\star(Y\star Z)
\]
We omit the subscript, writing $\Ass=\Ass_\star$, whenever the underlying product is clear.
All products considered in this paper are commutative, $X\star Y=Y\star X$.
If convenient, we interpret $\star$ as a $(1,2)$-tensor field.

\begin{defn}\label{defn:MFmfd}
    A \emph{\CMF manifold} is a Riemannian manifold $(M,g)$ together with a product
    \[
        \star:\mathfrak X(M)\times\mathfrak X(M)\to\mathfrak X(M)
    \]
    such that
    \begin{enumerate}
        \item $\star$ is commutative;
        \item $\star$ is compatible with $g$, i.e.\ the \emph{Frobenius property} holds,
        \begin{equation}\label{eq:Frobenius.condition}
            g(X\star Y,Z)=g(X,Y\star Z)
        \end{equation}
        where $X,Y,Z\in\mathfrak X(M)$;
        \item interpreting $\star$ as a map $\mathfrak X(M)\to\mathrm{End}(TM)$,
        \begin{equation}\label{eq:MFmfd.curv.cond}
            [\star(X),\star(Y)]=\mu\,R(X,Y)
        \end{equation}
        for a constant $\mu\in\mathbb R$, where $X,Y\in\mathfrak X(M)$ and where $R$ denotes the curvature tensor of $g$;
        \item $\nabla\star$ is symmetric (in its contravariant arguments).
    \end{enumerate}
    If $M$ is fixed, we refer to $(g,\star)$ as a \emph{\CMF structure}.
\end{defn}

Note that \eqref{eq:MFmfd.curv.cond} is a condition on the associator of $\star$. Indeed, due to commutativity, \eqref{eq:MFmfd.curv.cond} implies
\[
	\Ass_\star(X,Y,Z)=[\star(Z),\star(X)]Y=\mu R(Z,X)Y.
\]
Moreover, \eqref{eq:Frobenius.condition} implies
\[
	g(\star(Y)(X),Z) =g(X,\star(Y)(Z)).
\]
This means that the endomorphism $\star(Y)$ is self-adjoint with respect to $g$ for any $Y\in\mathfrak X(M)$. Note that $\star(Y)$ is the (left and right, due to commutativity) multiplication with~$Y$.

\noindent \CMF manifolds generalise the (flat) Frobenius manifolds defined in \cite{Manin1999}. To avoid confusion, we remark that these are more general than the Frobenius manifolds defined by Dubrovin \cite{Dubrovin1996}. In particular, note that the existence of a unit and of a Euler vector field are not imposed here. For the purposes of this paper, we refer to a \CMF manifold with flat (i.e.\ $R=0$) underlying metric~$g$ by the term \emph{Manin-Frobenius manifold}:
\begin{ex}
    We call a \CMF manifold whose underlying metric $g$ is flat a \emph{Manin-Frobenius manifold}. In this case, the underlying product $\star$ is commutative, associative and $g$-compatible. The symmetry of $\nabla\star$ is called the \emph{potentiality property}, e.g.\ \cite{Manin1999,Hertling2002_moduli}.
\end{ex}

\noindent There exists a vast literature on Manin-Frobenius manifolds. They appear in many contexts such as topological and quantum field theory \cite{Witten1991,DVV1991,Dubrovin1996,Dubrovin1998,Manin1999} and Gromov-Witten theory \cite{Hertling2002_moduli,HM2012_cohom}.
Examples with curvature also naturally arise in various contexts.
For instance, the condition~\eqref{eq:MFmfd.curv.cond} appears in the guise of the curved Witten-Dijkgraaf-Verlinde-Verlinde equation in \cite{KKLNS2017,KKLNS2018,Kozyrev2019}, establishing a correspondence between certain \CMF manifolds and $\mathcal N=4$ supersymmetric models on manifolds of dimension $n$.
\CMF manifolds also appear naturally in the theory of submanifolds. Specifically, \cite{Mokhov2008} establishes a duality relationship between \CMF manifolds and a special class of submanifolds, namely the so-called submanifolds with potential of normals.
A further application was identified by the author in \cite{Vollmer2025_Frobenius}, in the context of second-order superintegrable systems.
Examples of \CMF manifolds with non-vanishing curvature naturally arise in Hessian geometry.
\begin{defn}
    A Riemannian metric $g$ is called \emph{Hessian}, if there is a flat torsion-free connection $\Nabla$ such that, in a neighborhood of any point of $M$,
    \[
        g = \Nabla d\phi
    \]
    for a (locally defined) function $\phi$.
    The function $\phi$ is called \emph{Hessian potential}.
    A \emph{Hessian manifold} is a triple $(M,g,\Nabla)$ of this kind. If $M$ is fixed, we refer to $(g,\Nabla)$ as a \emph{Hessian structure} on $M$.
\end{defn}

We mention that in a Hessian manifold $(M,g,\Nabla)$ also the dual connection $\Nabla^*$ is flat, which follows from the \emph{fundamental theorem of information geometry}, cf.\ \cite{Nielsen2020}.

The following example is discussed in \cite{KKLNS2017}, see also the references therein.

\begin{ex}[\cite{KKLNS2017}]\label{ex:hessian}
    Let $(M,g,\Nabla)$ be a Hessian manifold. Since $\Nabla$ is flat and, locally, $g=\Nabla d\phi$ holds, we have that $\hat P=\nabla-\Nabla$ is a $(1,2)$-tensor field that satisfies the Codazzi equation, i.e.\ $\nabla\hat P$ is symmetric in its contravariant arguments. Indeed, the curvature tensor $R^D$ of $D$ satisfies
    \[
    	0 = R^{\Nabla}(X,Y) = R(X,Y)-d\hat P(X,Y)+[\hat P(X),\hat P(Y)],
    \]
    where $\hat P(X)(Y)=\hat P(X,Y)$. We now decompose this expression according to irreducible representations of the symmetric group. Projecting onto the part with algebraic curvature symmetry yields
    \[
        R(X,Y)=[\hat P(Y),\hat P(X)].
    \]
    It then follows immediately that $d\hat P(X,Y)=\nabla_X\hat P(Y)-\nabla_Y\hat P(X)=0$, which is the Codazzi equation.
    Moreover, one easily verifies that $P$, $P(X,Y,Z)=g(X,\hat P(Y,Z))$, is a totally symmetric $(0,3)$-tensor field with
    \[
    P=\frac12\,\Nabla^3\phi,
    \]
    locally where $\phi$ is defined.
    Therefore, $\star:=\hat P$ defines a commutative product on $\mathfrak X(M)$ that is $g$-compatible and satisfies the potentiality property. Its associator satisfies the condition in the definition of a \CMF manifold.
    We conclude that $(M,g,\Nabla=\nabla\pm\star)$ both satisfy the conditions of a \CMF manifold. For reasons that are going to become clear in Section~\ref{sec:Hesse.Frobenius}, we prefer the choice $(M,g,\Nabla=\nabla-\star)$.
\end{ex}

The paper is organised as follows:
In Section~\ref{sec:initial.observations}, we start with some preliminary considerations, such as the normalisation of \CMF structures and a characterisation via Lie triple systems.
In Section~\ref{sec:Hesse.Frobenius}, we consider the case of \CMF manifolds of constant sectional curvature that are consistent with their natural Hessian structure, cf.\ Definition~\ref{defn:HMF.mfd}. Theorem~\ref{thm:magic.formula} characterises these \CMF manifolds via a non-linear closed prolongation system. Corollary~\ref{cor:initial.data.HMF} uses this theorem to qualify the local initial data via commutative (non-)associative algebras that admit a compatible Frobenius inner product.
In Section~\ref{sec:characterisation}, we consider general \CMF manifolds, admitting any curvature. We show, qualitatively, that they are very close to being Hessian structures, see Theorem~\ref{thm:skew-hessian} and Corollary~\ref{cor:CMF.manifolds}.

As a first application, Section~\ref{sec:superintegrability} investigates certain second-order superintegrable systems. We obtain, cf.\ Theorem~\ref{thm:superintegrable.systems}, that examples on spaces of constant sectional curvature correspond 1-to-1 to Hesse-Frobenius structures.
A brief conclusion is given in Section~\ref{sec:conclusion}.

\section{Initial observations}\label{sec:initial.observations}

We begin with some basic findings that are useful for the characterization of \CMF manifolds. 
In particular, the following observation is a useful tool for the study of \CMF structures. Its proof is straightforward.
\begin{lem}\label{lem:normal.form}
    Let $(M,g,\star)$ be a \CMF manifold with
    \[
        [\star(X),\star(Y)]=\mu R(X,Y)\,,
    \]
    $\mu\in\mathbb R$. We define a new product $\circ$ by
    \begin{equation}\label{eq:normalization}
        \circ:=\begin{cases}
                0 & \text{if $\mu=0$} \\
                \frac{1}{\sqrt{|\mu|}}\star & \text{if $\mu\ne0$}
        \end{cases}\,.
    \end{equation}
    Then $(M,g,\circ)$ is a \CMF manifold with
    \[
        [\circ(X),\circ(Y)]=\varepsilon R(X,Y)\,,
    \]
    where $\varepsilon\in\{-1,0,+1\}$, $\varepsilon=\sign(\mu)$.
\end{lem}
\noindent The sign of $\mu$ is zero, if and only if $\mu=0$, and otherwise it is $\sign(\mu)=\frac{\mu}{|\mu|}$. We call $\varepsilon$ the \emph{signum} of the \CMF manifold.
Also, we say that $(M,g,\circ)$ is the \emph{normalization} of $(M,g,\star)$.

We stress again that the concept of Frobenius manifold adopted here does not require the existence of a unit or Euler vector field.

\begin{defn}
    We say that two \CMF manifolds are \emph{equivalent up to rescaling}, if their normalizations coincide. 
\end{defn}

\begin{ex}\label{ex:mu=0}
    The simplest case of \CMF manifolds is that with $\mu=0$. In this case, the endomorphisms $P_X:\mathfrak X(M)\to\mathfrak X(M)$, $P_XY=\hat P(X,Y)=X\star Y$, $X,Y\in\mathfrak X(M)$, always necessarily commute.
    Hence, $\star$ in Definition~\ref{defn:MFmfd} can be any commutative and associative product that is compatible with the metric $g$, and such that the potentiality property is satisfied. In particular, if and only if the metric is flat, the product $\star$ defines a Manin-Frobenius manifold with underlying metric $g$.
\end{ex}

We continue our exploration with the associator of \CMF manifolds.
Note that the associator $\Ass$ of $\star$ induces a map $\As:\Gamma(\Lambda^2(TM))\to\mathrm{End}(TM)$ via
\begin{equation}
    \As(X\wedge Y)(Z)=\Ass(X,Z,Y)\,.
\end{equation}
For our next lemma, we introduce Lie triple systems.
\begin{defn}\label{defn:Lie.triple.system}
    A \emph{Lie triple system} is a tri-linear map $C:V^3\to V$ on a (finite dimensional) vector space $V$ such that
    \begin{align*}
        C(u,v,w)&=-C(v,u,w)
        \\
        0&=C(u,v,w)+C(v,w,u)+C(w,u,v)
        \\
        C(u,v,C(w_1,w_2,w_3))&=C(C(u,v,w_1),w_2,w_3)
        \\ & \qquad +C(w_1,C(u,v,w_2),w_3)+C(w_1,w_2,C(u,v,w_3))
    \end{align*}
    for $u,v,w,w_1,w_2,w_3\in V$.
\end{defn}
We observe that the associator of $\star$ in a \CMF manifold is always skew-symmetric in its first and last arguments, and that it satisfies the cyclic condition.
Indeed, we have for a \CMF manifold $(M,g,\star)$ that the map $C:\mathfrak X(M)\times\mathfrak X(M)\times\mathfrak X(M)\to \mathfrak X(M)$,
\[
    C(X,Y,Z)=\As(X\wedge Y)(Z),
\]
satisfies the cyclic identity
\begin{equation}
    C(X,Y,Z)+C(Y,Z,X)+C(Z,X,Y)=0\,.
\end{equation}
Using \eqref{eq:MFmfd.curv.cond}, this corresponds to the Bianchi identity.
Now recall that the metric $g$ has constant sectional curvature, precisely if its curvature tensor has the form $R(X,Y)Z = \kappa\,(g(Y,Z)X-g(X,Z)Y)$ for some constant $\kappa$.
For \CMF manifolds of constant sectional curvature, we obtain the following:
\begin{lem}
    Let $(M,g,\star)$ be a normalised \CMF manifold of constant sectional curvature. Then the following two statements are true:
    \begin{enumerate}
        \item The map $C:\mathfrak X(M)\times\mathfrak X(M)\times\mathfrak X(M)\to \mathfrak X(M)$,
        \[
            C(X,Y,Z)=\As(X\wedge Y)(Z),
        \]
        defines a Lie triple structure.
        \item The map $\As$ takes values in $\mathfrak{so}(n)$, and it is a (maybe vanishing) constant multiple of the standard map $\Psi:\Gamma(\Lambda^2TM)\to\mathrm{End}(TM)$,
        \[
            \Psi: X\wedge Y\mapsto\left( Z\mapsto g(X,Z)Y-g(Y,Z)X \right)\,.
        \]
        It is trivial precisely if $g$ is flat.
    \end{enumerate}
\end{lem}
Note that here we tacitly identify $\mathfrak{so}(n)$ with the endomorphisms $S\in\mathrm{End}(TM)$ such that $g(X,SY)=-g(SX,Y)$.
\begin{proof}
	For the first part of the claim, the first two axioms in Definition \ref{defn:Lie.triple.system} have already been verified. The last axiom is then straightforwardly checked using \eqref{eq:MFmfd.curv.cond} together with the hypothesis of constant sectional curvature.
	The second part of the claim also immediately follows from \eqref{eq:MFmfd.curv.cond} together with the hypothesis of constant sectional curvature.
\end{proof}

\section{Hesse-Frobenius structures}\label{sec:Hesse.Frobenius}

We characterise \CMF manifolds that arise from Hessian structures.
In preparation, consider the tensor field $P\in\Gamma(\mathrm{Sym}_3(T^*M))$, $P(X,Y,Z)=g(X\star Y,Z)$.
Note that, because of commutativity, $P$ is symmetric in its first two arguments. It is then totally symmetric because of $g$-compatibility.
Since $\nabla\star$ is symmetric, it furthermore follows that $\nabla P$ is symmetric and, hence, that $P$ is a (rank-$3$) Codazzi tensor.

Solutions of rank-$3$ Codazzi equations are discussed in \cite{LSW1998,SSM2019}.
For example, in the case of a flat manifold, the tensor field $P$ is, locally, the third derivative of a (locally defined) function $\psi$, $P_{ijk}=\nabla^3_{ijk}\psi$.
Contrast this with the local existence of a Hessian potential $\phi$, which implies $P_{ijk}=\frac12\Nabla^3_{ijk}\phi$ on a (maybe smaller) local neighborhood.
Neither of these two potentials $\psi$ and $\phi$ is unique, but we may ask whether it is possible to choose~$\phi$ and~$\psi$ such that they coincide.
We will return to this question momentarily, but first we need to clarify further when a \CMF manifold arises as in Example~\ref{ex:hessian}.
\begin{lem}\label{lem:HMF}
    \begin{enumerate}
        \item\label{item:HMF.1} Let $(M,g,\star)$ be a \CMF manifold with $\mu<0$. Then its normalization $(M,g,\circ)$ arises as in Example~\ref{ex:hessian} from the Hessian structure $(g,\Nabla=\nabla-\circ)$.
        \item\label{item:HMF.2} Assume that $g$ is flat and that $(M,g,\star)$ is a \CMF manifold. For $r\in\mathbb R$, let $\circ=r\star$. Then $(M,g,\circ)$ arises as in Example~\ref{ex:hessian}, and in particular $(g,\Nabla=\nabla-\circ)$ is a Hessian structure.
        \item\label{item:HMF.3} Consider a \CMF manifold $(M,g,\star)$. If there is a real number $r\in\mathbb R$ such that $(M,g,\Nabla=\nabla-r\star)$ is Hessian, then $(M,g,\star)$ is as in the first or second number of this theorem.
    \end{enumerate}
\end{lem}
\noindent In short, Example~\ref{ex:hessian} essentially covers Manin-Frobenius manifolds as well as non-flat \CMF manifolds with $\mu<0$.
\begin{proof}
    We observe that, in Example~\ref{ex:hessian}, the \CMF structure is already in normalised form.
    
    We now prove claim~\ref{item:HMF.1}: for $\circ$ we have $[\circ(X),\circ(Y)]=-R(X,Y)$.
    Then the curvature of $\Nabla=\nabla-\circ$ is
    \[
    	R^{\Nabla}(X,Y)=R^g(X,Y)-d\circ(X,Y)+[\circ(X),\circ(Y)]=-d\circ(X,Y)=0,
    \]
    where in the last step we have used the potentiality condition for the \CMF structure $(g,\circ)$, i.e.\ that $\nabla\circ$ is symmetric.
    It follows that $\Nabla$ defines a Hessian structure.

    For the claim~\ref{item:HMF.2}, we proceed similarly. Note that $[\star(X),\star(Y)]=0$. Moreover, due to the flatness of $g$, $R^g=0$.
    Now consider $\circ=r\star$ for an arbitrary $r\in\mathbb R$. Let $\Nabla=\nabla-\circ$. Then
    \[
    	R^{\Nabla}(X,Y)=R^g(X,Y)-d\circ(X,Y)+[\circ(X),\circ(Y)]=-d\circ(X,Y)+r^2[\star(X),\star(Y)]=0,
    \]
    due to the potentiality property and because $r^2[\star(X),\star(Y)]=0$.

    For claim~\ref{item:HMF.3}, we set $\hat P=r\star$. By hypothesis, we have that
    \[
    	0 = R^{\Nabla}(X,Y)=R^g(X,Y)+d\hat P(X,Y)+[\hat P(X),\hat P(Y)]
    	= R^g(X,Y)+[\hat P(X),\hat P(Y)],
    \]
    and hence $[\hat P(X),\hat P(Y)]=-R^g(X,Y)$.
    If $R^g=0$, we are in the case \ref{item:HMF.2} and there is nothing to prove. Otherwise, we observe that $\circ=r\star$ is the normalisation of $\star$, and thus we are in the case \ref{item:HMF.1}.
\end{proof}

We now specialise to \CMF structures of constant sectional curvature. In this case, the Codazzi equation $\nabla_kP_{ij}^\ell-\nabla_jP_{ik}^\ell=0$ can be solved explicitly.
Indeed, from \cite{Ferus1981,LSW1998}, it follows that, locally, there is a function $\psi$ such that
\begin{equation}\label{eq:Frobenius.potential}
    P_{ijk}=\nabla^3_{(ijk)}\psi+4\kappa\,g_{(ij}\nabla_{k)}\psi
    =\nabla^3_{ijk}\psi
    +\kappa\,(2g_{ij}\nabla_k\psi+g_{ik}\nabla_j\psi+g_{jk}\nabla_i\psi)\,.
\end{equation}
Here, by $\nabla^3_{ijk}\psi$ we mean the triple covariant derivative in the order $\nabla_k\nabla_j\nabla_i\psi$, and round brackets denote symmetrisation in enclosed indices, $\nabla^3_{(ijk)}=\frac13(\nabla^3_{ijk}+\nabla^3_{jki}+\nabla^3_{kij})$.
\begin{lem}\label{lem:difference.of.potentials}
    Let $(M,g,\star)$ be a \CMF manifold of constant sectional curvature $\kappa$, and assume $\Nabla=\nabla-\star$ defines a Hessian structure $(g,\Nabla)$.
    Then the Hessian potential~$\phi$,
    \[
        g_{ij}=(\Nabla d\phi)_{ij}\,,
    \]
    and the Frobenius potential $\psi$,
    \[
        P_{ijk}=\nabla^3_{ijk}\psi
        +\kappa\,(2g_{ij}\nabla_k\psi+g_{ik}\nabla_j\psi+g_{jk}\nabla_i\psi)\,,
    \]
    satisfy, in a neighborhood where both are simultaneously defined,
    \begin{equation}\label{eq:difference.of.potentials}
        \nabla^3_{ijk}(\psi+\phi)
        = \left(
        	\hat P_{ij}^b\hat P_{bk}^a 
        	-\nabla_k\hat P_{ij}^a
        \right)\nabla_a\phi
        -\kappa\,(2g_{ij}\nabla_k\psi+g_{ik}\nabla_j\psi+g_{jk}\nabla_i\psi)\,.
    \end{equation}
    Here, we write $\hat P=\star$ for better legibility and we define $P$ by $P_{ijk}=g_{ka}\star_{ij}^a$.
\end{lem}
\begin{proof}
    From the hypothesis, we obtain, locally where the Hessian potential $\phi$ is defined,
    \[
	    2P_{ijk} = \Nabla_kg_{ij} = \Nabla^3_{ijk}\phi\,.
    \]
    We then compute
    \begin{equation*}
    	2P_{ijk}
    	= \Nabla^3_{ijk}\phi
    	= \Nabla_k\left( \nabla^2_{ij}\phi+\hat P_{ij}^a\nabla_a\phi\right)
    	= \nabla^3_{ijk}\phi+3P_{ijk}
    	+\left(
    	\nabla_k\hat P_{ij}^a
    	-\hat P_{bk}^a\hat P_{ij}^b
    	\right) \nabla_a\phi\,,
    \end{equation*}
    from which we infer
    \begin{equation}\label{eq:magic.P.formula}
    	-P_{ijk} = \nabla^3_{ijk}\phi
    	+\left(
    	\nabla_k\hat P_{ij}^a
    	-\hat P_{ij}^b\hat P_{bk}^a
    	\right) \nabla_a\phi\,.
    \end{equation}
    On the other hand, see \eqref{eq:Frobenius.potential}, we have
    \begin{equation}\label{eq:Frobenius.potential.formula}
	    P_{ijk}=\nabla^3_{ijk}\psi
    	+\kappa\,(2g_{ij}\nabla_k\psi+g_{ik}\nabla_j\psi+g_{jk}\nabla_i\psi)\,,
    \end{equation}
    where $\psi$ is the Frobenius potential. Indeed, since $P$ is symmetric and satisfies the Codazzi equation, i.e.\ $\nabla P$ is totally symmetric, this follows from \cite{Ferus1981}, for instance.
    Combining both expressions for $P$, we conclude that
    \begin{equation}\label{eq:magic.P.formula.pre.reduced}
    	0 = \nabla^3_{ijk}(\phi+\psi)
    	+\left(
    	\nabla_k\hat P_{ij}^a-\hat P_{ij}^b\hat P_{bk}^a
    	\right)\nabla_a\phi
    	+\kappa\,(2g_{ij}g_k^a+g_{ik}g_j^a+g_{jk}g_i^a)\nabla_a\psi\,.
    \end{equation}
    The claim follows.
\end{proof}

Note that neither $\phi$ nor $\psi$ are uniquely determined. Indeed, the Hessian potential is unique only up to the addition of $\Nabla$-affine terms. Similarly, $\psi$ can be altered by adding any solution $\chi$ of
\[
    \nabla^3_{ijk}\chi
        +\kappa\,(2g_{ij}\nabla_k\chi+g_{ik}\nabla_j\chi+g_{jk}\nabla_i\chi)=0\,,
\]
since this does not change the underpinning \CMF structure.
It is thus a natural question whether the Hessian potential $\phi$ and the Frobenius potential $\psi$ can be chosen to coincide.
In order to answer this question, we introduce for \CMF manifolds a concept of consistency with a Hessian structure.

\begin{defn}\label{defn:HMF.mfd}
    We say that a \CMF manifold $(M,g,\star)$ is a \emph{\HMF manifold} of constant sectional curvature $\kappa$, if there exists a Hessian structure $(g,\Nabla)$ such that
    \begin{equation}
        \star:\mathfrak X(M)\times\mathfrak X(M)\to\mathfrak X(M)\,,\qquad
        X\star Y:=\nabla_XY-\Nabla_XY\,.
    \end{equation}
    and such that, locally around any given point, any Hessian potential $f$ of $(g,\Nabla)$ satisfies
    \[
        P_{ijk}=-\nabla^3_{(ijk)}f-4\kappa\,g_{(ij}\nabla_{k)}f\,,
    \]
    where round brackets indicate symmetrization in the enclosed indices.
\end{defn}

\noindent \HMF manifolds are rich geometric structures. They are interesting as geometric structures in their own right. We will also explain later how they naturally appear in mathematical physics, and more specifically in the context of superintegrable Hamiltonian systems.
Examples can also be found in supersymmetric mechanics \cite{KKLNS2017,KKLNS2018} and submanifold theory \cite{Mokhov2008}.
Definition~\ref{defn:HMF.mfd} however has limitations as it requires the underlying metric to have constant sectional curvature. This requirement ensures the existence of the Frobenius potential via the solution of the Codazzi equation described in \cite{Ferus1981,LSW1998}. Yet one quickly realises that whenever there is an explicit solution of $\nabla_kP_{ijl}-\nabla_jP_{ikl}=0$ of the form
\[
    P_{ijk} = \nabla^3_{ijk}\psi+\Theta_{ijk}^a\nabla_a\psi,
\]
where $\Theta$ depends only on the data $(g,P)$, and not on the choice of $\psi$, then a condition similar to the above one can be obtained.

We now characterise \CMF manifolds of constant sectional curvature that arise from Hessian structures.

\begin{thm}\label{thm:magic.formula}
    Assume that $g$ is of constant sectional curvature $\kappa$.
    In a \HMF manifold $(M,g,\star)$, the product satisfies the first-order prolongation system
    \begin{equation}\label{eq:HMF.condition}
        \nabla_k\star_{ij}^\ell=\star_{ij}^a\star_{ak}^\ell
        +\kappa\,(2g_{ij}g_k^\ell+g_{ik}g_j^\ell+g_{jk}g_i^\ell)\,.
    \end{equation}
    Conversely, if the latter condition holds for a commutative and $g$-compatible product $\star$, and if $\nabla\star$ is symmetric, then $(M,g,\star)$ is a \HMF structure.
\end{thm}
\begin{proof}
    In a \HMF manifold $(M,g,\star)$, we have that the (flat) connection $\Nabla=\nabla-\star$ defines the Hessian structure $(M,g,\Nabla)$.
    Again, we introduce the auxiliary objects $\hat P=\star$ and $P_{ijk}=g_{ka}\star_{ij}^a$.
    Hence, locally, there is a function $\phi$ such that
    $g=\Nabla d\phi$ and $2P_{ijk}=\Nabla^3_{ijk}\phi$
    with \eqref{eq:magic.P.formula}.
    Moreover, there exists a local Frobenius potential $\psi$ such that \eqref{eq:Frobenius.potential.formula} holds.
    Without loss of generality, we assume that $\phi$ and $\psi$ are defined on the same domain.
    In Lemma~\ref{lem:difference.of.potentials} we have seen that
    \begin{equation}\label{eq:magic.P.formula.pre.reduced.repeated}
        0 = \nabla^3_{ijk}(\phi+\psi)
	        +\left(
	            \nabla_k\hat P_{ij}^a-\hat P_{ij}^b\hat P_{bk}^a
	        \right)\nabla_a\phi
	        +\kappa\,(2g_{ij}g_k^a+g_{ik}g_j^a+g_{jk}g_i^a)\nabla_a\psi\,.
    \end{equation}
    By hypothesis, $(M,g,\star)$ is a Hesse-Frobenius manifold, and therefore $\psi=-\phi$. We thus infer
    \begin{equation}\label{eq:magic.P.formula.reduced}
    	0 = \left(
	    	\nabla_k\hat P_{ij}^a-\hat P_{ij}^b\hat P_{bk}^a
	    	-\kappa\,(2g_{ij}g_k^a+g_{ik}g_j^a+g_{jk}g_i^a)
	    \right)\nabla_a\phi\,.
    \end{equation}
    Now, observe that $\phi$ is determined up to $\Nabla$-affine changes, as it is a Hessian potential. Hence, by changing the Hessian potential accordingly, we obtain a basis of $\Omega^1(M)$. It follows that
    \begin{equation*}
        \nabla_k\hat P_{ij}^a =
        \hat P_{ij}^b\hat P_{bk}^a
        +\kappa\,(2g_{ij}g_k^a+g_{ik}g_j^a+g_{jk}g_i^a)\,.
    \end{equation*}
    This proves the direction ``$\Rightarrow$'' of the claim.
    
    Conversely, assuming this latter identity, note that \eqref{eq:MFmfd.curv.cond} follows under the hypothesis. Hence, $(g,\star)$ defines a \CMF structure. Resubstituting \eqref{eq:HMF.condition} into~\eqref{eq:difference.of.potentials} (i.e.\ into \eqref{eq:magic.P.formula.pre.reduced.repeated}) we arrive at
    \begin{equation*}
        \nabla^3_{ijk}(\psi+\phi) = 
        -\kappa\,(2g_{ij}\nabla_k\phi+g_{ik}\nabla_j\phi+g_{jk}\nabla_i\phi)
        -\kappa\,(2g_{ij}\nabla_k\psi+g_{ik}\nabla_j\psi+g_{jk}\nabla_i\psi)\,,
    \end{equation*}
    and therefore
    \begin{equation*}
        \nabla^3_{ijk}(\psi+\phi) = 
        -\kappa\,(2g_{ij}\nabla_k(\psi+\phi)+g_{ik}\nabla_j(\psi+\phi)+g_{jk}\nabla_i(\psi+\phi))\,.
    \end{equation*}
    For any Hessian potential $\phi$, we observe that $\psi=-\phi$ is a solution. Hence, the structure is \HMF.
\end{proof}

\begin{rmk}
    Note that our definition of a \HMF structure requires any Hessian potential to also be a Frobenius potential. If we only require to find one such Hessian potential $\phi$, it has to be such that
    \begin{equation}\label{eq:weak.HMF.condition}
        \left( \nabla_k\star_{ij}^\ell -\star_{ij}^a\star_{ak}^\ell
                -\kappa\,(2g_{ij}g_k^\ell+g_{ik}g_j^\ell+g_{jk}g_i^\ell) \right)(d\phi)_\ell = 0\,.
    \end{equation}
    This guarantees that
    \[
        \nabla^3_{ijk}(\psi+\phi)
        +\kappa\,(2g_{ij}\nabla_k(\psi+\phi)+g_{ik}\nabla_j(\psi+\phi)+g_{jk}\nabla_i(\psi+\phi))=0\,,
    \]
    and hence $\psi-(\psi+\phi)=-\phi$ is also an admissible Frobenius potential.
    Note that, if the data $(g,\kappa,\star)$ and a Hessian potential $\phi'$ in $\Nabla$-affine coordinates are known, \eqref{eq:weak.HMF.condition} is a linear condition on the $\Nabla$-affine correction to $\phi'$ that yields $\phi$.
\end{rmk}

Let us now turn our attention to \eqref{eq:HMF.condition}. We characterise the initial data from which suitable structures can be locally obtained.

\begin{cor}\label{cor:initial.data.HMF}
    Consider a Riemannian manifold $(M,g)$ of constant sectional curvature $\kappa$. Moreover, let $p\in M$ and let $b=g_p$ be the metric evaluated at $p$. Assume that $(T_pM,+,\diamond)$ is an algebra such that $+$ is the usual addition on $T_pM$ and
    \begin{enumerate}
        \item $\diamond$ is commutative;
        \item $\diamond$ is $b$-compatible, i.e.\ $b(u\diamond v,w)=b(u,v\diamond w)$ for $v,v,w\in T_pM$;
        \item the associator of $\diamond$ satisfies
        $\Ass_\diamond(u,v,w)=\kappa \left( g(u,v)w-g(w,v)u \right)$.
    \end{enumerate}
    Then there is a neighborhood $U\subset M$ of $p$ with a Hessian structure
    \[
        (U,g|_U,\Nabla)\,,
    \]
    such that $(U,g,\star=\nabla-\Nabla)$ defines a \HMF manifold.
\end{cor}
We remark that an inner product $b$ that satisfies the above compatibility condition is often called a \emph{Frobenius inner product}.
\begin{proof}
    Interpreting $\star$ as a $(1,2)$-tensor field, we write, for better legibility,
    \[
    	\hat P=\star.
    \]
    We make use of the Ricci identity for \eqref{eq:HMF.condition}, this is
    \[
    	\nabla_k\nabla_l\hat P_{ij}^a-\nabla_l\nabla_k\hat P_{ij}^a
    	= R^b_{ikl}\hat P_{bj}^a+R^b_{jkl}\hat P_{ib}^a+\hat P_{ij}^bR^a_{bkl}\,.
    \]
    Now, differentiating \eqref{eq:HMF.condition} with respect to $\nabla$ and antisymmetrising, we obtain
    \begin{align*}
    	\nabla_k\nabla_l\hat P_{ij}^a-\nabla_l\nabla_k\hat P_{ij}^a
    	&- R^b_{ikl}\hat P_{bj}^a - R^b_{jkl}\hat P_{ib}^a - \hat P_{ij}^bR^a_{bkl}
    	\\
    	&=\nabla_k\hat P_{ij}^b\ \hat P_{bl}^a
    	+\cancel{ \hat P_{ij}^b\ \nabla_k\hat P_{bl}^a }
    	-\nabla_l\hat P_{ij}^b\ \hat P_{bk}^a 
    	-\cancel{ \hat P_{ij}^b\ \nabla_l\hat P_{bk}^a }
    	\\
    	&\qquad\qquad\qquad
    	- R^b_{ikl}\hat P_{bj}^a - R^b_{jkl}\hat P_{ib}^a - \hat P_{ij}^bR^a_{bkl}
    	\\
    	&= \hat P_{ij}^c\left(
    			\hat P_{ck}^b\hat P_{bl}^a - \hat P_{cl}^b\hat P_{bk}^a
    	 \right)
    	 - R^b_{ikl}\hat P_{bj}^a - R^b_{jkl}\hat P_{ib}^a - \hat P_{ij}^bR^a_{bkl}
    	\\
    	&\qquad\qquad\qquad
    	-\kappa\left(
    	g_{il}\hat P_{jk}^a+g_{jl}\hat P_{ik}^a-g_{ik}\hat P_{jl}^a-g_{jk}\hat P_{il}^a
    	\right)
    	=0
    \end{align*}
    Hence, at $p$, the integrability condition for \eqref{eq:HMF.condition} is always satisfied due to the (non-)associa\-tivity of $\diamond$.
    It remains to verify that \eqref{eq:Frobenius.condition} and \eqref{eq:MFmfd.curv.cond} are preserved by solutions to \eqref{eq:HMF.condition}.
    This is straightforwardly verified by taking a derivative with respect to $\nabla$ and using \eqref{eq:HMF.condition} to replace derivatives of $\hat P$, also employing \eqref{eq:Frobenius.condition} and \eqref{eq:MFmfd.curv.cond} to simplify the resulting expressions.
    We conclude that \eqref{eq:HMF.condition} admits a unique solution for any initial data $\diamond$ as in the hypothesis, cf.~\cite{Goldschmidt1967} for instance.
    This means that the initial data can be extended in a small neighborhood $U\ni p$ to a product $\star$ such that $(g|_U,\star)$ defines a \HMF structure on $(U,g|_U)$, proving the claim.
\end{proof}

Corollary~\ref{cor:initial.data.HMF} characterises \HMF manifolds as \CMF manifolds whose product is determined by its data in a single point, extended locally in a way that ensures that the Frobenius and the Hessian potential are, up to sign, identical.

\section{Characterizing \CMF manifolds}\label{sec:characterisation}

We conclude our study with a characterization of \CMF manifolds of positive signum.
In preparation, consider a \CMF manifold $(M,g,\star)$ such that $\star$ is not associative and $\mu>0$. Then it is a solution of
\begin{align*}
    P_{ijk} &= P_{(ijk)} \\
    \nabla_k P_{ij}^\ell &= \nabla_{(k}P_{ij)}^\ell \\
    P_{ik}^a P_{ja}^\ell-P_{jk}^a P_{ia}^\ell &= R_{ijk}^b
\end{align*}
and vice versa.
Recall that the $g$-dual connection of $\Nabla$ is defined by~\eqref{eq:dual.connection}.

\begin{defn}\label{defn:skew-hessian}
    For the purposes of this paper, we call a triple $(M,g,\Nabla)$ consisting of a Riemannian manifold $(M,g)$ and a torsion-free connection $\Nabla$ a \emph{skew-Hessian manifold}, if
    \begin{enumerate}
        \item the $g$-dual connection $\Nabla^*$ of $\Nabla$ is torsion-free as well,
        \item the curvature of $\Nabla$ is non-vanishing, and
        \item the curvature of $\Nabla$ is twice the curvature of $g$.
    \end{enumerate}
\end{defn}

This provides us with a characterization of the remaining cases of \CMF manifolds

\begin{thm}\label{thm:skew-hessian}
    \CMF manifolds with $\mu=1$ correspond to skew-Hessian structures. More precisely: 
    \begin{enumerate}
        \item A skew-Hessian manifold defines a \CMF manifold with $\mu=1$ via
        \[
            \star=\Nabla-\nabla.
        \]
        \item A \CMF manifold with $\mu=1$ is a skew-Hessian manifold, if its underlying metric is non-flat.
    \end{enumerate}
\end{thm}

\begin{proof}
    For the first claim, we observe that the torsion-freeness of $\Nabla$ and $\Nabla^*$ implies that $P$ defined by $P(X,Y,Z)=g(\hat P(X,Y),Z)=g(\Nabla_XY-\nabla_XY,Z)$ for $X,Y,Z\in\mathfrak X(M)$, is totally symmetric, which implies that $\star$ is commutative and $g$-compatible.
    Note that the curvature tensor of~$\Nabla$ decomposes according to irreducible representations of the symmetric group.
    Hence, the condition that the curvature tensor of $\Nabla$ is a non-zero multiple of $R^g$,
    \[
        R^\Nabla=\lambda R^g\,,\quad \lambda\ne0\,,
    \]
    decomposes into two independent conditions:
    firstly, $P$ has to be a Codazzi tensor with respect to $\nabla$, and hence $\star$ satisfies the potentiality property.
    Secondly, we have (with $X,Y\in\mathfrak X(M)$)
    \[
        R^g(X,Y)+[\star(X),\star(Y)]=2\,R^g(X,Y).
    \]
    Therefore, \eqref{eq:MFmfd.curv.cond} is satisfied with $\mu=1$.

    Under the hypothesis of the second claim, we have
    \[
        [\star(X),\star(Y)]=R^g(X,Y)\,,
    \]
    which does not vanish.
    We define $\Nabla=\nabla+\star$.
    Since $P$ is totally symmetric, $\Nabla$ and $\Nabla^*$ are torsion-free.
    Moreover, we find
    \[
        R^\Nabla(X,Y)=R^g(X,Y)+[\star(X),\star(Y)]=2\,R^g(X,Y)
    \]
    and hence that $R^\Nabla$ is non-vanishing. The theorem is therefore proved.
\end{proof}

We obtain the following characterization of \CMF manifolds.
\begin{cor}\label{cor:CMF.manifolds}
    A \CMF manifold falls into precisely one of the following cases
    \begin{itemize}
        \item its normalization is a Hessian manifold in the sense of Example~\ref{ex:hessian}, or
        \item its underlying metric is non-flat and $\star$ is associative, or
        \item its normalization is a skew-Hessian manifold.
    \end{itemize}
\end{cor}
\begin{proof}
    This follows immediately from Example~\ref{ex:mu=0}, together with Lemma~\ref{lem:HMF} and Theorem~\ref{thm:skew-hessian}.
\end{proof}

As a consequence, it appears justified to say that there exist few qualitatively different types of \CMF manifolds. Indeed, up to rescalings, \CMF manifolds are either Hessian structures, skew-Hessian structures or they are non-flat and associative. Note that the first condition in Definition~\ref{defn:skew-hessian} makes $(M,g,\Nabla)$ a \emph{statistical manifold} \cite{Nielsen2020,Shima2007}. Recall that a statistical manifold is a Riemannian manifold equipped with a pair $\Nabla,\Nabla^*$ of $g$-dual connections.
The typical example of a \CMF manifold thus is a statistical manifold such that the curvature tensors of $\nabla$ and $\Nabla$ (or of $\nabla$ and $\Nabla^*$) are the same up to multiplication by a constant.

\section{Application to second-order superintegrability}\label{sec:superintegrability}

Applications of \CMF manifolds have been described across the literature in various contexts.
Flat (Manin-)Frobenius structures are ubiquitous in the literature on topological and quantum field theory, e.g.\ \cite{Witten1991,DVV1991,Dubrovin1996,Dubrovin1998,Manin1999}. They are also linked to Gromov-Witten theory \cite{Hertling2002_moduli,HM2012_cohom}, for instance.
Examples with non-vanishing curvature are less frequent, but also appear naturally in applications in physics and geometry \cite{KKLNS2017,KKLNS2018,Kozyrev2019,Mokhov2008}.

We report a further class of examples here, namely second-order superintegrable systems. Our discussion below is based on \cite{Vollmer2025_Frobenius} and \cite{CV2025} as well as \cite{AV2025,KSV2023,KSV2024}.
Consider a Riemannian manifold $(M,g)$. The cotangent space $T^*M$ carries a natural symplectic structure $\omega$. We denote by $(x,p)$ canonical Darboux coordinates on $T^*M$. Then the map $H:T^*M\to\mathbb R$,
\[
    H(x,p)=g^{-1}(p,p)+V(x)
\]
is called \emph{Hamiltonian} with potential $V:M\to\mathbb R$.
The vector field $X_H$ on $T(T^*M)$ with $\omega(X_H,-)=dH$ is called the \emph{Hamiltonian vector field} for $H$.

A (maximally) \emph{superintegrable system} is defined by a collection of functions $F^{(\nu)}:T^*M\to\mathbb R$, $1\leq\nu\leq2n-2$, such that
\begin{enumerate}
    \item $(H,F^{(1)},\dots,F^{(2n-2)})$ is functionally independent
    \item each $F^{(\nu)}$ is a constant of the motion for $H$, i.e.\ $X_H(F^{(\nu)})=0$.
\end{enumerate}
A superintegrable system is called \emph{second-order} if the constants of the motion can be chosen in the form (we drop the Einstein convention temporarily)
\[
    F^{(\nu)}(x,p)=\sum_{i,j=1}^n K^{(\nu)\,ij}(x)p_ip_j+W^{(\nu)}(x),
\]
where we assume $K^{(\nu)\,ij}=K^{(\nu)\,ji}$ without loss of generality.
One can show that the coefficients $K^{(\nu)\,ij}$ define components of Killing tensors
\[
    K_{(\nu)}=\sum_{i,j=1}^n \left( \sum_{a,b=1}^n g_{ai}g_{bj}K^{(\nu)\,ab}\right)\,dx^i\otimes dx^j\,.
\]
A second-order superintegrable system is called \emph{irreducible}, if the endomorphisms associated to these Killing tensors via $g$ do not share a common invariant subspace.

For irreducible second-order superintegrable systems, there is a tensor field $T=\sum_{i,jk=1}^n T_{ij}^k\,dx^i\otimes dx^j\otimes \partial_k$ which is symmetric in $(i,j)$ and trace-free with respect to $g$ in $(i,j)$, such that (we use the Einstein convention again)
\[
    \nabla^2_{ij}V = T\indices{_{ij}^k}\nabla_kV+\frac1n\,g_{ij}\,\Delta V
\]
Here, $\Delta$ is the Laplace-Beltrami operator.
We define the one-form $t$ via
\begin{equation}\label{eq:t.from.T}
    t(X)=\mathrm{tr}(Y\to T(X,Y))
\end{equation}
and moreover the totally symmetric tensor field
\begin{equation}\label{eq:Tt2P}
    P(X,Y,Z)=\frac13\,\left(g(T(X,Y),Z)+\frac{1}{n-1}\,g(X,Y)t(Z)\right).
\end{equation}
By contraction, this formula immediately implies
$\mathrm{tr}_g(P(-,-,Z))=\frac{n}{3(n-1)}\,t(Z)$, and hence
\begin{equation}\label{eq:P2T}
	g(T(X,Y),Z)=3\,P(X,Y,Z)-\frac3n\,g(X,Y)\,\mathrm{tr}_g(P(-,-,Z)),
\end{equation}
such that the data of $P$ and that provided by $T$ can be used interchangeably.
In \cite{KSV2023,KSV2024,KSV2024_2D}, the structural equations for irreducible second-order superintegrable systems are obtained for the case when all integrability conditions for $V$ and $K_{(\nu)}$ are satisfied generically (these systems are called \emph{abundant}). For $n\geq3$, it is found in \cite{KSV2023,KSV2024} that under this abundantness assumption and for $g$ of constant sectional curvature $\kappa$, $P$ is totally symmetric and satisfies
\begin{multline}\label{eq:sis.alg}
    g^{-1}(P(X,Y),P(Z,W))-g^{-1}(P(X,Z),P(Y,W)) \\
    = -\kappa\,(g(X,Y)g(Z,W)-g(X,Z)g(Y,W))
\end{multline}
for any vector fields $X,Y,Z,W$. Moreover, the tensor $T$ is subject to the differential conditions (in index notation, $\circ$ denotes the trace-free part)
\begin{subequations}\label{eq:sis.diff}
    \label{eq:prolongation:T}
    \begin{align}
        \label{eq:DT:DS}
        \mathring{T}_{ijk,l}
            &=\frac1{13}\,\Pi_{(ijk)_\circ}
            \bigg[
                \mathring{T}\indices{_{ij}^a}\mathring{T}_{kla}
                + \mathring{T}_{ijk}\bar t_l
                +3\,\mathring{T}_{ijl}\bar t_k
                \notag \\
            &\qquad\qquad\qquad\qquad +
                \left(
                    \frac4{n-2}\,\mathring{T}\indices{_i^{ab}}\mathring{T}_{jab}
                    -3\,\mathring{T}\indices{_{ij}^a}\bar t_a
                \right)g_{kl}
            \bigg]\\
        \label{eq:DT:Dt}
        \bar t_{k,l}
            &=\frac13
              \left(
                -\frac2{n-2}\mathring{T}\indices{_k^{ab}}\mathring{T}_{lab}
                +3\mathring{T}\indices{_{kl}^a}\bar t_a
                +4\bar t_k\bar t_l
              \right)_\circ
              \notag \\
            &\qquad
                +\frac1ng_{kl}\,\left(
                    \frac{3n+2}{6(n+2)(n-1)}\,\mathring{T}^{abc}\mathring{T}_{abc}
                    -\frac{n-2}6\,\bar t^a\bar t_a
                    +\frac{3}{2(n-1)}\,R
                \right)
    \end{align}
\end{subequations}
where $\Pi_{(ijk)_\circ}$ is the projector onto the symmetric trace-free part with respect to the enclosed indices, and where $\bar t=\frac{n}{(n-1)(n+2)}\,t$. For details see \cite{KSV2023}.

Now assume that $(M,g)$ is a Riemannian manifold of dimension $n\geq3$ and of constant sectional curvature and that $P$ is a totally symmetric, cubic tensor field on $M$ satisfying~\eqref{eq:sis.alg} and~\eqref{eq:sis.diff}, recalling \eqref{eq:P2T} and \eqref{eq:t.from.T}.
Let $X\star Y:=P(X,Y,-)$ for vector fields $X,Y$ on $M$.
It was shown in~\cite{AV2025} that $\Nabla=\nabla-\star$ defines a Hessian structure $(g,\Nabla)$ and that, for any Hessian potential~$\phi$,
\[
    P=\Nabla^3\phi.
\]
Hence, we are in the case of a \HMF structure and conclude
\begin{equation}\label{eq:sis.diff.better}
    \nabla_l C_{ijk}=C_{ija}g^{ab}C_{bkl}
        +\kappa\,(2g_{ij}g_{kl}+g_{ik}g_{jl}+g_{jk}g_{il})\,.
\end{equation}

\begin{rmk}
    Note that \eqref{eq:sis.diff.better} implies both~\eqref{eq:sis.alg} and~\eqref{eq:sis.diff}. We remark that the form \eqref{eq:sis.diff.better} can also be obtained similarly to the respective computation in \cite{Vollmer2025_Frobenius}, where the flat case is discussed (i.e.\ the case of Manin-Frobenius manifolds).
\end{rmk}

We therefore obtain:
\begin{thm}\label{thm:superintegrable.systems}
    On Riemannian manifolds of constant sectional curvature,
    abundant second-order superintegrable systems correspond 1-to-1 to \HMF structures.
\end{thm}

\section{Conclusion}\label{sec:conclusion}

We found that, possibly up to rescaling, \CMF manifolds fall into three distinct types, namely Hessian structures, skew-Hessian structures and (non-flat) spaces with a commutative and associative, metric-compatible product.
These three types all fit under the broader umbrella of statistical manifolds.

In particular, we have investigated the case of \HMF structures, which describe the case of \CMF structures that are consistent with a Hessian structure in the sense that the Hessian potentials serve simultaneously as Frobenius potentials. We have characterised \HMF structures intrinsically via a stronger potentiality property, which also encodes the associator-curvature condition \eqref{eq:MFmfd.curv.cond}.
This case is both geometrically natural and relevant in mathematical physics, which we illustrated with an example from the theory of second-order maximally superintegrable systems. We have also seen that these examples can be reconstructed, locally around a point on the manifold, from purely algebraic data.

\section*{Acknowledgements}

I would like to thank John Armstrong, Olaf Lechtenfeld and Andrey Konyaev for discussions. I thank the University of Hannover, the Forschungsfonds of the Department of Mathematics at the University of Hamburg and the SFB/CRC 1624: Higher structures, moduli spaces and integrability at the University of Hamburg for financial support.

\renewcommand*{\bibfont}{\footnotesize\small}
\printbibliography

\end{document}